\documentclass{amsart}
\usepackage{amsmath,amssymb,amsthm, mathrsfs}
\usepackage{boxedminipage}
\usepackage{graphicx}
\usepackage[margin=1in]{geometry}
\usepackage{mathpazo}
\usepackage{fancyhdr}
\usepackage[utf8]{inputenc}
\usepackage{color}
\usepackage{tikz}
\usepackage{hyperref}

\newtheorem{theorem}{Theorem}[section]
\newtheorem{lemma}[theorem]{Lemma}

\newtheorem{proposition}[theorem]{Proposition}

\newtheorem{varexample}[theorem]{Example}
\newtheorem{definition}[theorem]{Definition}

\newcommand{\gon}{\textrm{gon}}

\newenvironment{example}{\begin{varexample}
\begin{normalfont}}{\end{normalfont}
\end{varexample}}

\theoremstyle{definition}

\definecolor{gold}{rgb}{1.0,0.75,0.0}

\title{The Gonality of Queen's Graphs}

\author{Ralph Morrison and Noah Speeter}

\begin{document}

\maketitle

\begin{abstract}
    In this paper we study queen's graphs, which encode the moves by a queen on an $n\times m$ chess board, through the lens of chip-firing games.  We prove that their gonality is equal to $nm$ minus the independence number of the graph, and give a one-to-one correspondence between maximum independent sets and classes of positive rank divisors achieving gonality.  We also prove an identical result for toroidal queen's graphs.
\end{abstract}

\section{Introduction}

The Eight Queens Puzzle, dating back to at least 1848 \cite{first_puzzle}, asks whether eight queens may be placed on the squares of a standard \(8\times 8\) chessboard such that no queens may attack another (meaning that no two are in a common row or column, or on a common diagonal).  More generally, one can consider \(n\times n\) chessboards.  It was shown in \cite{first_solution1,first_solution2} that for \(n\geq 4\), there is indeed always a solution to the puzzle.  One can ask similar  question for an \(n\times n\) toroidal chessboard, which allow for more diagonal moves; the situation is more subtle here, with a solution of \(n\) queens precisely when \(\gcd(6,n)=1\) \cite{first_toroidal}.  We refer the reader to \cite{survey_queens} for an excellent survey of the history of the problem and the mathematical results that have been proved. 

It is natural to frame this puzzle in terms of graph theory.  Given \(n,m\geq 2\), we define the \(n\times m\) queen's graph \(Q_{n,m}\) to have \(nm\) vertices, organized into \(m\) columns of \(n\) vertices each and \(n\) rows of \(m\) vertices each, with two vertices connected precisely if they are in the same row or column, or lie along a diagonal of slope \(\pm 1\).  We similarly define the \(n\times m\) toroidal queen's graph \(TQ_{n,m}\), where diagonal moves are allowed across the boundary of a torus.  The graphs \(Q_{3, 3}\) and \(TQ_{3,3}\) are illustrated in Figure \ref{figure:lambda_3_case}. 
The placement of many mutually non-attacking queens translates to a question of the \emph{independence number} of a graph $G$, denoted $\alpha(G)$, defined as the largest possible size of a set of vertices with no two sharing an edge. In particular, the Eight Queens Puzzle is equivalent to asking whether $\alpha(Q_{8,8})\geq 8$.

\begin{figure}[hbt]
    \centering
    \includegraphics[scale=1]{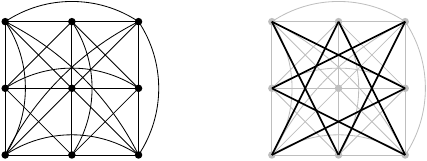}
    \caption{The graph \(Q_{3,3}\) on the left and \(TQ_{3,3}\) on the right; the edges in \(TQ_{3,3}\) not present in \(Q_{3,3}\) are highlighted.}
    \label{figure:lambda_3_case}
\end{figure}

In this paper we study queen's graphs from the perspective of chip-firing games on graphs.  In particular, we study the \emph{gonality} of these graphs, which measures the smallest number of chips that can be placed on a graph so that any added debt can be eliminated through chip-firing moves.  For \(Q_{n,m}\), we obtain the following complete answer in terms of the independence number \(\alpha(G)\).

\begin{theorem}\label{theorem:queens}
Let \(m\geq n\geq 2\).  We have \(\textrm{gon}(Q_{n,m})=nm-\alpha(Q_{n,m})\).  More explicitly,
\[\textrm{gon}(Q_{n,m})=\begin{cases}3&\textrm{ if }n=m=2\\7&\textrm{ if }n=m=3\\n(m-1)&\textrm{ otherwise. }\end{cases}\]
Moreover, there is a one-to-one correspondence between divisor classes \([D]\) on \(Q_{n,m}\) of rank \(1\) and degree \(\gon(Q_{n,m})\), and maximum independent sets on \(Q_{n,m}\).
\end{theorem}

The idea behind this proof is that on any graph \(G\), any independent set \(S\) of maximum size gives a winning strategy with \(|V(G)|-|S|\) chips.  For queen's graphs, it turns out that not only is this strategy optimal; every optimal strategy is of this form, and no two maximum independent sets give chip-firing-equivalent strategies.

For toroidal queen's graphs, we find a similar result.

\begin{theorem}\label{theorem:toroidal}
We have \(\gon(TQ_{n,m})= nm-\alpha(TQ_{n,m})\).  In particular, for \(n=m\), we have 
\[\gon(TQ_{n,n})= \begin{cases}n(n-1)&\textrm{ if }\gcd(6,n)=1\\n(n-1)+1&\textrm{ if }3\nmid n \textrm{ and } 4\nmid n\\n(n-1)+2&\textrm{ in all other cases,}\end{cases}\]
and \(\gon(TQ_{n,m})=nm-\gcd(n,m)\) for \(n\neq m\).  Moreover,  there is a one-to-one correspondence between divisor classes \([D]\) on \(Q_{n,m}\) of rank \(1\) and degree \(\gon(TQ_{n,m})\), and maximum independent sets on \(TQ_{n,m}\).
\end{theorem}

Our paper is organized as follows.  In Section \ref{section:background} we present background results on graph theory and queen's graphs, as well as the definitions necessary for chip-firing games.  In Section \ref{section:queens} we prove Theorem \ref{theorem:queens}, and in Section \ref{section:toroidal_queens} we prove Theorem \ref{theorem:toroidal}.

\medskip

\textbf{Acknowledgements.}  The authors thank Sasha Kononova and anonymous reviewers for helpful comments on earlier drafts of this work. The first author was supported by NSF Grant DMS-2011743.

\section{Graphs and chip-firing games}
\label{section:background}
\subsection{Graphs}
For this paper we assume all graphs are connected and without loops. Given a graph $G$ we denote the set of its vertices with $V(G)$ and the set of its edges with $E(G)$. Given a subset $U\subseteq V(G)$, we denote its complement as $U^c$. 

Given a vertex $v\in V(G)$, we say the \emph{degree} of $v$, $\textit{deg}(v)$ is the number of edges adjacent to $v$. The minimal degree of a graph $G$ is denoted $\delta(G)=\min\{\textit{deg}(v_i)|v_i\in V(G)\}$.

Given a graph $G$, we say $A\subseteq V(G)$ is \emph{independent} if for all $v_i, v_j \in A$, $v_i$ and $v_j$ are not adjacent. The \emph{independence number} of $G$, denoted $\alpha(G)$, is the maximum possible size of an independent set on $G$.  We recall several previous results regarding the independence numbers of queen's graphs and their variants.

\begin{lemma} We have \(\alpha(Q_{2,2})=1\), \(\alpha(Q_{3,3})=2\), and \(\alpha(Q_{n,n})=n\) for \(n\geq 4\).  For \(n\neq m\), \(\alpha(Q_{n,m})=\min\{n,m\}\).
\end{lemma}

\begin{proof}  The upper bound of \(\min\{n,m\}\) (whether \(n=m\) or not) follows from the fact that no row or column can have a pair of non-adjacent vertices.  The lower bound of \(n\) on \(\alpha(Q_{n,n})\) when \(n\geq 4\) is the work of \cite{first_solution1,first_solution2}, which gives  a lower bound of \(n\) on \(\alpha(Q_{n,m})\)  when \(4\leq n\leq m\).  The cases of \(Q_{2,2}\) and \(Q_{3,3}\) can be determined by brute force; and \(\alpha(Q_{2,3})=2\) and \(\alpha(Q_{3,4})=3\) by construction of explicit independent sets. Finally, for \(n\in \{2,3\}\) and \(n<m\), having \(Q_{n,n+1}\) as a subgraph of \(Q_{n,m}\) gives the desired lower bound.
\end{proof}

\begin{lemma}\cite{toroidal_square1,toroidal_square2}  We have \[\alpha(TQ_{n,n})= \begin{cases}n&\textrm{ if }\gcd(6,n)=1\\n-1&\textrm{ if }3\nmid n \textrm{ and } 4\nmid n\\n-2&\textrm{ in all other cases.}\end{cases}\]
\end{lemma}

\begin{lemma}[\cite{non_square_tori}] We have \(\alpha(TQ_{n,m})=\gcd(n,m)\) for \(n\neq m\).
\end{lemma}

    A \emph{cut} on a graph is a partition of the vertices into two sets $(U, U^c)$. Given the cut $(U,U^c)$, the \emph{cutset} is the set of edges $\{(v_i,v_j)|v_i\in U, v_j\in U^c\}$ and is denoted $E(U,U^c)$.

\subsection{Chip-firing}
We now describe divisor theory on graphs, as well as the language of chip-firing games.  For a thorough treatment, we refer the reader to \cite{divisors_and_sandpiles}.  Given a graph $G$, a \emph{divisor} is an integer vector in $\mathbb{Z}^{V(G)}$, or alternatively, a $\mathbb{Z}-$linear combination of the vertices of $G$.  Divisors are often described as stacks of poker chips on each vertex of a graph. A negative number associated to a vertex is interpreted as that vertex having a debt of chips.  Given a vertex subset $U\subset V(G)$, we let $\mathbb{1}_U$ denote the divisor that places $1$ chip on each vertex in $U$, and $0$  chips elsewhere.

The name \emph{chip-firing} comes from the action of firing vertices. We can fire the vertex $v$ by transferring one chip along each edge connected to $v$ in the direction away from $v$. Given a divisor $D$, we may obtain some other divisor $D'$ by firing some sequence of vertices.  This process is reversible; for instance, to undo a firing at \(v\), we may fire all vertices in \(\{v\}^C\).  Differing by a sequence of chip-firing move thus forms an equivalence relation on the set of all divisors on \(G\), and we say \(D\) and \(D'\) are \emph{equivalent} if they are some sequence of chip-firing moves apart.  We let $[D]$ denote the equivalence class of a divisor $D$.

Note that the order in which we fire multiple vertices will not change the resulting divisor. Thus it is often useful to talk about firing a subset \(U\) of vertices simultaneously, which simply means to fire all vertices in \(U\) in any order.  Firing the subset \(U\) is called \emph{legal} if no vertices of \(U\) are in debt, and no vertices of \(U\) go into debt from the subset-firing move.

Given a divisor $D$ on a graph $G$, the \emph{degree} of $D$ is the sum of all integer coordinates of the vector $D$. Alternatively the degree of a divisor can be thought of as the sum of all chips and chip debts that exist on the graph.   A divisor $D$ is \emph{effective} if no coordinates of the vector are negative. In other words, a divisor is effective if no vertex has a debt.
   
\begin{example}
   \begin{figure}[h]

\begin{tikzpicture}
\filldraw (0,1) circle (2pt);
\node at (0,1.3) {2};
\filldraw (1,0) circle (2pt);
\filldraw (1,2) circle (2pt);
\node at (1,2.3) {-1};
\node at (1.5,1) {$D$};
\node at (1,-0.3) {-1};
\node at (3,1.3) {2};
\filldraw (2,0) circle (2pt);
\filldraw (2,2) circle (2pt);
\filldraw (3,1) circle (2pt);
\draw (0,1)--(1,0);
\draw (0,1)--(1,2);
\draw (1,0)--(1,1);
\draw (1,0)--(2,0);
\draw (1,1)--(1,2);
\draw (1,2)--(2,2);
\draw (2,0)--(2,1);
\draw (2,0)--(3,1);
\draw (2,1)--(2,2);
\draw (2,2)--(3,1);
\filldraw (4,1) circle (2pt);
\filldraw (5,0) circle (2pt);
\filldraw (5,2) circle (2pt);
\filldraw (6,0) circle (2pt);
\filldraw (6,2) circle (2pt);
\filldraw (7,1) circle (2pt);
\node at (6,-0.3) {1};
\node at (6,2.3) {1};
\node at (5.5,1) {$D'$};
\draw (4,1)--(5,0);
\draw (4,1)--(5,2);
\draw (5,0)--(5,1);
\draw (5,0)--(6,0);
\draw (5,1)--(5,2);
\draw (5,2)--(6,2);
\draw (6,0)--(6,1);
\draw (6,0)--(7,1);
\draw (6,1)--(6,2);
\draw (6,2)--(7,1);
\end{tikzpicture}
\caption{Two equivalent divisors $D$ and $D'$;  $D'$ is obtained from $D$ by firing both vertices which had two chips.}
\label{equivalent divisors}
\end{figure}
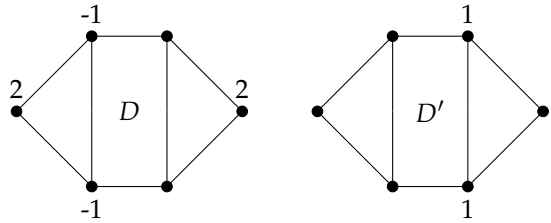
In Figure \ref{equivalent divisors}, we see two equivalent divisors of degree $2$. The divisor $D'$ is effective while $D$ is not.
\end{example}

We now discuss \(q\)-reduced divisors.  Given a vertex \(q\in V(G)\), we say that \(D\) is \emph{\(q\)-reduced} if \(D(v)\geq 0\) for all \(v\neq q\), and there exists no subset \(U\subset V(G)-\{q\}\) that can be legally fired. Every divisor \(D\) is equivalent to a unique \(q\)-reduced divisor, denoted $D_q$ \cite[Theorem 3.7]{divisors_and_sandpiles}.  Given a divisor \(D\) with \(D(v)\geq 0\) for all \(v\neq q\), we can test whether \(D\) is \(q\)-reduced through \emph{Dhar's burning algorithm} \cite{dhar}. This algorithm allows a ``fire'' to spread throughout the graph as follows:
\begin{itemize}
\item[(1)] Let \(q\) burn.
\item[(2)] Any edge incident to a burning vertex will also burn.
\item[(3)] Any vertex \(v\neq q\) with more than \(D(v)\) burning edges incident to it will burn.
\end{itemize}
This process will stabilize either with the whole graph burned, or with some collection \(U\) of unburned vertices.  Then \(D\) is \(q\)-reduced if and only if the whole graph burns.  In fact, if the set of unburnt vertices \(U\) is nonempty, then \(U\) will give a legal subset firing move.  Firing \(U\) and then repeating the process until the whole graph burns will eventually yield \(D_q\).

We present the following lemma to help understand how the burning process from Dhar's algorithm manifests in graphs with many subgraphs isomorphic to a complete graph \(K_n\), which has $n$ vertices all adjacent to one another.

\begin{lemma}\label{lemma:burning_complete} Let \(H\) be a subgraph of \(G\) with \(H\) isomorphic to \(K_n\) for some \(n\geq 1\).  Suppose that \(D\) is an effective divisor on \(G\) with  \(D(q)=0\) for some \(q\in V(H)\), such that \(H\) has at most \(n-1\) chips.  If Dhar's burning process is run from \(q\), and all of \(H\) does not burn, then either some \(v\in V(H)\) satisfies \(D(v)\geq n-1\), or every \(v\in V(H)-\{q\}\) satisfies \(D(v)\geq 1\).  The same holds if \(n\geq 5\) and we have at most \(n\) chips on \(H\).
\end{lemma}

\begin{proof} This claim holds immediately if $n\leq 3$, due to the limited number of ways chips could be placed.  Now assume $n\geq 4$.  Let \(U\) be the set of unburned vertices in \(H\), with \(k=|U|\). Every vertex in \(U\) has at least \(n-k\) chips on it, since there are that many burning edges coming from other vertices in \(H\).  Thus there are at least \(k(n-k)\) chips on \(H\).  Suppose \(2\leq k\leq n-2\).  Since \(k(n-k)\) is concave down as a function of \(k\), this is minimized when \(k=2\) or when \(k=n-2\), both giving a value of \(2n-4\).  Thus there are at least \(2n-4\) chips on \(H\), a contradiction since \(2n-4>n-1\) for \(n\geq 4\).  We conclude that \(k=1\) or \(k=n-1\), yielding the required configuration of chips.

For \(n\geq 5\), the same argument goes through with \(n\) chips on \(H\), since \(2n-4>n\) for \(n\geq 5\).
\end{proof}

\subsection{Rank and gonality}
We now define the \emph{rank} \(r(D)\) of a divisor \(D\). If $D$ is not equivalent to any effective divisor, then \(r(D)=-1\).  Otherwise, for \(k\geq 0\), we define \(r(D)\geq k\) if for any effective divisor $E$ of degree $k$, $D-E$ is equivalent to some effective divisor.  In chip-firing language, \(r(D)\) is the maximum amount of added debt \(D\) can remove through chip-firing moves, regardless of where that debt is placed.  Using language of rank, we now present the key graph invariant of this paper.

\begin{definition}
    Given a graph $G$, the \emph{gonality} of $G$, written $\textit{gon}(G)$, is the minimum degree of a positive rank divisor on $G$. 
\end{definition}

In the language of chip-firing games, the gonality is the minimum number of chips we can place on the graph such that, wherever \(-1\) debt is added, we can eliminate it using chip-firing moves.  An equivalent characterization is that \(\gon(G)\) is the minimum degree of an effective divisor \(D\) such that \(D_q(q)\geq 1\) for all \(q\in V(G)\).

\begin{example}
We claim that the graph \(G\) in Figure \ref{equivalent divisors} has \(\gon(G)=2\).  Indeed, the divisor \(D'\) has positive rank, as can be checked by computing \(D_q\) for all \(q\in V(G)\), implying \(\gon(G)\leq 2\).  To see \(\gon(G)>1\), note that any effective divisor \(D\) of degree \(1\) will be \(q\)-reduced for all \(q\in V(G)\) (including those with \(D(q)=0\)), as can be checked by Dhar's burning algorithm.
\end{example}

This example highlights the sorts of strategies that can be used for bounding gonality above and below:  upper bounds come from divisors of positive rank, and lower bounds can be found using arguments based on \(q\)-reduced divisors and Dhar's burning algorithm. We have a general upper bound in the following lemma.

\begin{lemma}[Proposition 3.1 in \cite{gonality_of_random_graphs}]\label{lemma:upper_bound_alpha} For a simple graph \(G\), we have \(\gon(G)\leq |V(G)|-\alpha(G).\)
\end{lemma}
We can explicitly describe a positive rank divisor of this degree:  choose an independent set of size \(\alpha(G)\), and place a single chip on every vertex outside of that set. If debt is placed on any unchipped vertex \(v\), firing \(V(G)-\{v\}\) will eliminate the debt.

Our strategy in the later sections for determining the gonality for queen's graphs and toroidal queen's graphs will be an upper bound coming from Lemma \ref{lemma:upper_bound_alpha} and a lower bound coming from a Dhar's burning algorithm argument. 
 To help with the latter, we introduce the following  terminology.  Let \(G\) be a graph with \(Q_{n,m}\) as a spanning subgraph. Considering an effective divisor on such a graph, we may refer to the \emph{poorest row} as the row with the smallest number of chips.
We call an effective divisor \(D\) on \(G\) \emph{row-equitable} if, among all effective divisors equivalent to \(D\), the divisor \(D\) maximizes the number of chips on its poorest row.

Our last lemma shows that distinct choices of maximum independent sets give non-equivalent divisors.

\begin{lemma}\label{lemma:pairwise_inequivalent} Let $G$ be $Q_{n,m}$ or $TQ_{n,m}$, and let $S$ and $T$ be maximum independent sets on $G$.  Then $\mathbb{1}_{S^C}\sim \mathbb{1}_{T^C}$ if and only if $S=T$.
\end{lemma}

\begin{proof} The reverse direction is immediate.  For the forward direction, assume  $\mathbb{1}_{S^C}\sim \mathbb{1}_{T^C}$.  Choose a vertex $q\in V(G)$ such that $q$ is in the same row or column as an element of $S$, and the same row or column as an element of $T$.  We will show that $\mathbb{1}_{S^C}$ and $\mathbb{1}_{T^C}$ are both $q$-reduced; by uniqueness of $q$-reduced divisors, this means that $\mathbb{1}_{S^C}\sim \mathbb{1}_{T^C}$ implies $\mathbb{1}_{S^C}= \mathbb{1}_{T^C}$ and thus $S=T$.

We will show that $\mathbb{1}_{S^C}$ is $q$-reduced; an identical argument holds for $\mathbb{1}_{T^C}$.  We will also assume that $S$ has a vertex $v$ in the same row as $q$; the argument holds just as well if it is in the same column instead.  Start Dhar's burning algorithm on $\mathbb{1}_{S^C}$ from $q$, which immediately burns the unchipped vertex $v$.  From here each vertex in that row will burn, since each has at most one chip but is incident to two burning edges.  Now we claim that once a row burns, the vertices in the rows above and below it (if they exist) will burn as well.  This is because such a vertex has two or more neighbors in the burning row, but only one chip.  Thus once one row burns, the rows above and below burn as well.  This spreads the fire throughout the whole graph until every vertex is burned, implying $\mathbb{1}_{S^C}$ is $q$-reduced.  This completes the proof.
\end{proof}

\section{Gonality of the queen's graph}
\label{section:queens}

Throughout this section, we assume that $m\geq n\geq 2$.  We start with the following lemma, which gives us Theorem \ref{theorem:queens} for a few small queen's graphs, namely $Q_{2,2}$ and $Q_{3,3}$.  These are the only queen's graphs that do not satisfy $\alpha(Q_{n,m})=n$, so it is convenient to handle them separately

\begin{lemma}\label{lemma:q22_q33} We have that $\gon(Q_{2,2})=3$ and $\gon(Q_{3,3})=7$.  Morever, for both these graphs, there is a one-to-one correspondence between divisor classes of rank $1$ of degree equal to gonality and  maximum independent sets.
\end{lemma}

\begin{proof}

Consider the graph \(Q_{2,2}\).  This is simply the complete graph \(K_4\), which has gonality \(3\). (In general, $\gon(K_n)=n-1$ for $n\geq 2$; this can be viewed as a corollary of Lemma \ref{lemma:burning_complete}, or can be proved by considering treewidth as in \cite[Example 4.3]{debruyn2014treewidth}.) By Lemma \ref{lemma:burning_complete}, we can completely characterize the degree \(3\), rank \(1\) effective divisors as ones that place \(1\) chip on three vertices, or \(3\) chips on one vertex; and in fact every divisor of the latter form is equivalent to a divisor of the former form, by firing the vertex with \(3\) chips.  These are precisely the divisors of the form $\mathbb{1}_{S^C}$ for a maximum independent set $S$ (i.e., a set consisting of a single vertex), as desired.  These divisors  $\mathbb{1}_{S^C}$ are pairwise non-equivalent by Lemma \ref{lemma:pairwise_inequivalent}, giving us the desired one-to-one correspondence.

Now consider the graph \(Q_{3,3}\).  Note that each corner vertex has degree \(6\); that the other boundary vertices also have degree \(6\); and that the central vertex has degree \(8\).  Thus we have that the minimum degree \(\delta(Q_{3,3})\) is at least as large as \(\lfloor |V(Q_{3,3})|/2\rfloor +1=5\).  This allows us to apply \cite[Corollary 3.2]{Ech} to conclude that \(\gon(Q_{3,3})=|V(Q_{3,3})|-\alpha(Q_{3,3})=9-2=7\).

We now classify all degree \(7\) divisors of positive rank, up to equivalence.  Let \(D\) be a row-equitable divisor of degree \(7\) and positive rank on \(Q_{3,3}\).  Choose a vertex \(q\) with \(D(q)=0\), and run Dhar's burning algorithm from \(q\); since \(r(D)>0\), we know there must be an unburned set of vertices \(U\), with \(1\leq |U|\leq 8\). We remark that \(|U|=1\) is impossible: it would require a single vertex to have at least \(6\) chips, which cannot happen in a single row (much less a single vertex) by the row-equitable assumption.

We know that \(|E(U,U^C)|\leq 7\), since firing \(U\) must be a legal firing move.  We have \(|E(U,U^C)|\geq 6|U|-2{|U|\choose 2}=7|U|-|U|^2\), since each vertex has degree at least \(6\), and we have to subtract off twice any edges that may be internal to \(U\).
A symmetric argument shows that \(|E(U,U^C)|=|E(U^C,U)|\geq 6|U^C|-2{|U^C|\choose 2}=7|U^C|-|U^C|^2\).
Note that if \(|U|=2\) or \(|U|=5\), the first inequality  gives $|E(U,U^C)|\geq10$; the same holds for \(2\leq |U|\leq 5\) since the expression \(7|U|-|U|^2\) is concave down. This is a impossible since $|E(U,U^C)|\leq 7$.
Similarly, when \(|U|=6\) or \(|U|=7\), the second inequality gives a lower bound of \(12\) and \(10\) respectively on $|E(U,U^C)|$, which is again impossible.  Thus we must have $|U|=8$, so the burning process must stop when one vertex burns. This holds for any choice of \(q\) with \(D(q)=0\), meaning that the set of unchipped vertices is an independent set. Thus \(D=\mathbb{1}_{S^C}\) for some independent set \(S\), which must be maximum since \(\deg(D)=7=9-\alpha(Q_{3,3})\).

Therefore, up to equivalence, every rank \(1\) degree \(7\) divisor on \(Q_{3,3}\) is of the form \(\mathbb{1}_{S^C}\) for some maximum independent set \(S=\{u,v\}\).  These divisors are pairwise inequivalent by Lemma \ref{lemma:pairwise_inequivalent}.
\end{proof}

We now present the key lemma for lower bounding the gonality of queen's graphs.

\begin{lemma}\label{lemma:deg(D)_lower_bound} Let \(G\) be a graph with the \(n\times m\) queen's graph \(Q_{n,m}\) as a spanning subgraph.  Let \(D\) be a row-equitable effective divisor of positive rank on \(G\) with \(\deg(D)\leq nm-1\), and let \(q\) be a vertex with \(D(q)=0\). Suppose that when running Dhar's burning algorithm on \(D\) starting from \(q\), a poorest row burns. Then \(\deg(D)\geq (m-1)n+2\).
\end{lemma}

\begin{proof}  Since \(r(D)>0\), we know that the whole graph does not burn.  We will argue that some row is completely unburned.  Suppose for the sake of contradiction that every row has at least one vertex burn. We can fire our unburnt vertices $U$ to obtain a new effective divisor $D'$.  We will consider the poorest row \(r\) under this new divisor.

Suppose \(r\) has no vertices in \(U\).  Then \(r\) gained chips, meaning that the poorest row of \(D'\) has more chips than the poorest row of \(D\), a contradiction to our choice of \(D\).

Suppose \(r\) has only one vertex in \(U\), say \(v\in U\).  That row must have at least \(m-1\) chips in \(D'\), received from firing \(v\); and as the poorest row it must have at most \(m-1\) chips, so \(r\) has exactly \(m-1\) chips in \(D'\).  It follows that \(U=\{v\}\):  otherwise \(r\) would have gained an additional chip from a vertex outside of the row, giving \(r\) more than \(m-1\) chips.  It follows that \(D(v)\geq \deg(v)>(n-1)+(m-1)\) chips, one to donate to each neighbor; the strict inequality comes from the fact that \(v\) has at least one diagonal neighbor in \(G\).  Since \(D\) maximized the number of chips in the poorest row, and \(D'\) has \(m-1\) chips in the poorest row, so too must \(D\).  It follows that every row in \(D\) had at least \(m-1\) chips, so the most chips a single row could contain in \(D\) is
\[(nm-1)-(n-1)(m-1)=n+m-2.\]
But \(v\) contains strictly more chips, a contradiction. A similar argument holds if \(r\) has \(m-1\) vertices in \(U\), with each of those \(m-1\) vertices requiring more than \(1+(n-1)=n\) vertices to donate to their neighbors in \(U^C\), again meaning \(r\) has more chips in \(D\) than is possible.

We are now left to find a contradiction in the case that the poorest row \(r\) of \(D'\) has between \(2\) and \(m-2\) vertices in \(U\).
If row $k$ has $j$ vertices in $U$, then firing $U$ will transfer $j(m-j)$ chips from the $j$ vertices in $U$ to the other $m-j$ vertices in row $k$. Since $D'$ is effective, row $k$ will have at least $j(m-j)$ chips in $D'$. For every row that is not totally burned, we have $1< j< m-1$, implying that $j(m-j)> m-1$. Thus \(r\) could not be the poorest row, giving us our final contradiction.

We now know that at least one row remains unburned. We then consider the cut set $E(U,U^c)$. Since both $U$ and $U^c$ contain an entire row, they both contain at least $m\geq n-1$ vertices and by \cite[Theorem 5.1]{rooks_graph_gonality}, $E(U, U^c)$ has at least $(n-1)m$ non-diagonal cut edges. The first and last rows of a queen's graph share $2(m-n+1)$ diagonal edges between them. Any other two rows will share more diagonal edges and thus,
     \[|E(U,U^c)|\geq (n-1)m+2(m-n+1)\]
     \[=nm-m+2m-2n+2\]
    \[=nm+m-2n+2\]
    \[\geq nm-n+2\]
    \[=(m-1)n+2.\]
Since \(\deg(D)\geq E(U,U^c)\), this completes the proof.
\end{proof}

We now prove a proposition that does most of the work towards proving Theorem \ref{theorem:queens}.

\begin{proposition}\label{prop:independent_sets} Let \(D\) be a row-equitable divisor on \(Q_{n,m}\) with \(\deg(D)=\gon(Q_{n,m})\) and \(r(D)>0\). Then \(D=\mathbb{1}_{S^C}\) for some maximal independent set \(S\).
\end{proposition}

\begin{proof} We have already seen this result holds when \(n=m=2\) and when \(n=m=3\), so we will assume we are in neither case.  This means that \(\alpha(Q_{n,m})=n\).

Suppose for the sake of contradiction that \(D\) is a row-equitable divisor with \(\deg(D)=\gon(Q_{n,m})\) and \(r(D)>0\), such that \(D\) is not of this form.  It will suffice to show that there exists a vertex \(q\) in the poorest row of \(D\) with $D(q)=0$ such that running Dhar's on \(q\) makes that row burn.  This allows us to apply Lemmas \ref{lemma:upper_bound_alpha} and \ref{lemma:deg(D)_lower_bound} to conclude 
\[(m-1)n+2\leq \deg(D)=\gon(Q_{n,m})\leq nm-\alpha(Q_{n,m})=(m-1)n,\]
a contradiction.

If the poorest row in \(D\) has at most \(m-2\) chips, then choosing any \(q\) in that row with \(D(q)=0\) causes that row to burn by Lemma \ref{lemma:burning_complete}.  Otherwise the poorest row has \(m-1\) chips, meaning that since \(\deg(D)=\gon(Q_{n,m})\leq (m-1)n\), every row has exactly \(m-1\) chips. Thus, any row can serve as the poorest row, and so it suffices to show that a single row burns.

We will assume that every row has either \(1\) chip on every vertex except one; or all \(m-1\) chips on a single vertex.  (By Lemma \ref{lemma:burning_complete} these are the only configurations that avoid a copy of \(K_m\) burning immediately with \(m-1\) chips.) We split into two cases.
\begin{itemize}
    \item[(1)] Suppose some row \(r_i\) has a vertex \(u\) with \(m-1\) chips. Run Dhar's burning algorithm on some \(q\neq u\) in the same row; then all of \(r_i-u\) burns. Choose an adjacent row \(r_k\). Each vertex in \(r_k\) is adjacent to at least two vertices in \(r_i\), so any unchipped vertex (of which there is at least one) in \(r_k\) will burn.  From here, any vertex with only one chip will burn, as it will have at least two burning neighbors (one from \(r_k\) and one from \(r_i\)).  If all of \(r_k\) burns, we are done.  If it does not burn internally, then some vertex has \(m-1\) chips; but this has \(m-1\) burning edges from \(r_k\) and at least one burning edge from \(r_i\), so that vertex burns as well, and we have a burnt row.
\item[(2)] Suppose every row has \(m-1\) chips on different vertices, leaving exactly one vertex per row unchipped. 
 In particular, we have that every vertex in the graph has at most one chip.  By our supposition on \(D\), we may choose \(q,v\) adjacent in \(Q_{n,m}\) such that \(D(q)=D(v)=0\).  Run Dhar's burning algorithm from \(q\), so that \(v\) also burns.  If \(q\) and \(v\) are adjacent along a diagonal edge, then \(q\) has another neighbor in \(v\)'s row, which must have \(1\) chip and will therefore burn from \(q\) and \(v\); from here that row burns entirely.  If \(q\) and \(v\) are adjacent along a vertical edge, then the column \(c_\ell\) containing \(q\) and \(v\) will burn, as it is a copy of \(K_n\) with two burning vertices with at most \(1\) chip per vertex. From here even a single additional vertex burning makes a whole row burn; but any vertex in a column neighboring \(c_\ell\) will burn, as it has at most \(1\) chip and at least two burning edges from \(c_\ell\).
\end{itemize}

In all cases, a row burns, completing the proof.
\end{proof}

\begin{proof}[Proof of Theorem \ref{theorem:queens}] We may assume our graph is not $Q_{2,2}$ or $Q_{3,3}$, as these were handled in Lemma \ref{lemma:q22_q33}. Let \(D\) be a divisor on \(Q_{n,m}\) with \(\deg(D)=\gon(Q_{n,m})\) and \(r(D)>0\).  By Proposition \ref{prop:independent_sets}, \(D\) is equivalent to a divisor of the form \(\mathbb{1}_{S^C}\) for some maximal independent set \(S\), meaning that \(\gon(Q_{n,m})=\deg(D)=|S^C|=nm-\alpha(Q_{n,m})\).

 For the second claim, given a maximum independent set \(S\) on \(Q_{n,m}\), map \(S\) to the divisor class \([D]\) where \(D=\mathbb{1}_{S^C}\).  This is a positive rank divisor of degree \(\gon(Q_{n,m})\).   The fact that this map is onto the set of all positive rank divisor classes with degree equal to gonality follows from Proposition \ref{prop:independent_sets}. The fact that this map is one-to-one follows from Lemma \ref{lemma:pairwise_inequivalent}.
 \end{proof}

It follows from Theorem \ref{theorem:queens} that \(G=Q_{n,m}\) has as many pairwise inequivalent divisors \(D\) with \(r(D)>0\) and \(\deg(D)=\gon(G)\) as it has independent sets. This number is given by OEIS sequence A000170 \cite{oeis}. For instance, the graph \(Q_{8,8}\) has \(92\) pairwise non-equivalent divisors of degree \(56\) and positive rank.

 \section{Toroidal queen's graphs}
 \label{section:toroidal_queens}

As in the previous section we assume $m\geq n\geq 2$.  We begin with the following lemma, which plays a similar role to Lemma \ref{lemma:deg(D)_lower_bound} for toroidal queens' graphs.

 \begin{lemma}\label{lemma:deg(D)_bound_toroidal} Let \(G=TQ_{n,m}\), let \(D\) be a row-equitable effective divisor of positive rank on \(G\) with \(\deg(D)\leq nm-1\), and let \(q\) be a vertex with \(D(q)=0\). When running Dhar's burning algorithm on \(D\) starting from \(q\), it is impossible for a poorest row to burn.
 \end{lemma}

 \begin{proof} Suppose for the sake of contradiction that a poorest row does burn.  Since $r(D)>0$, we know the whole graph does not burn. Let $U$ be the set of unburnt vertices. Following the proof of Lemma \ref{lemma:deg(D)_lower_bound} until the analysis of \(E(U,U^C)\), we find that at least one row remains unburned.  In the remainder of the proof, we will argue that \(|E(U,U^C)|\geq nm\), yielding a contradiction since $\deg(D)\geq |E(U,U^C)|$ and $\deg(D)\leq nm-1$.

 Label the vertices of \(G\) with integer coordinates \((x,y)\) where \(0\leq x\leq m-1\) and \(0\leq y\leq n-1\).  Then the number of positive diagonal edges from \((0,0)\) to row \(i\) is equal to precisely the number of ordered pairs of the form \((k\mod m,k\mod n)\) where \(k\in\mathbb{Z}\) with \(k\equiv i\mod n\).  This can be viewed as starting with \((i,i)\) and adding \((n,n)\) to it repeatedly, working in the group \(\mathbb{Z}_m\oplus \mathbb{Z}_n\).  The number of ordered pairs we find is thus the order of \(n\) in \(\mathbb{Z}_m\), which is \(m/\gcd(n,m)\).  Thus \((0,0)\) is adjacent to at least \(m/\gcd(n,m)\) vertices in the \(i^{th}\) row.  The same argument holds for every vertex in the \(0^{th}\) row, so there are at least \(m^2/\gcd(n,m)\) edges connecting the \(0^{th}\) row with the \(i^{th}\) row. By the symmetry of the graph, the same holds for any pair of rows.

 Now we consider \(E(U,U^C)\).  There are at least \((n-1)m\) cut edges coming from non-diagonal edges; as well as at least \(m^2/\gcd(n,m)\) positive diagonal edges from the burned row to the unburned row.  However, some diagonal edges may duplicate a vertical one, so we must subtract the number of vertical edges between those two rows, namely \(m\).  Thus we have
 \[E(U,U^C)\geq (n-1)m+m^2/\gcd(n,m)-m=nm+m\left(\frac{m}{\gcd(n,m)}-2\right).\]
 In the case that \(m>n\) we have \(\gcd(n,m)<m\), implying that \(\frac{m}{\gcd(n,m)}\geq 2\).  Thus  \(E(U,U^C)\geq nm\), as desired.  If \(n=m\), each vertex is incident to at least \(1\) vertex in every other row along a diagonal edge, giving
 \[E(U,U^C)\geq (n-1)n+n=n^2,\]
 completing the proof.
 \end{proof}

We now prove our key proposition for toroidal queen's graphs.

\begin{proposition}
\label{prop:independent_sets_toroidal} Let \(m\geq n\), and let \(D\) be a row-equitable divisor on \(TQ_{n,m}\) with \(\deg(D)=\gon(Q_{n,m})\) and \(r(D)>0\). Then \(D=\mathbb{1}_{S^C}\) for some maximal independent set \(S\).
\end{proposition}

\begin{proof} First we note that the claim holds if \(TQ_{n,m}=K_{nm}\), so we will assume for the remainder of the proof that \(\alpha(TQ_{n,m})\geq 2\).  Note that this means that \(\gcd(n,m)>1\).

Suppose for the sake of contradiction that \(D\) is not of the form \(\mathbb{1}_{S^C}\) for a maximum independent set \(S\).  Note that \(\deg(D)\leq nm-\alpha(TQ_{n,m})\) by Lemma \ref{lemma:upper_bound_alpha}, and that there are two unchipped adjacent vertices \(q,v\in V(TQ_{n,m})\).

We will show that we may run Dhar's burning algorithm and burn a whole poorest row, leading to a contradiction by Lemma \ref{lemma:deg(D)_bound_toroidal}.  If a poorest row has at most \(m-2\) chips, we may start the algorithm at any unchipped vertex in that row and burn the row immediately by Lemma \ref{lemma:burning_complete}.  Thus we may assume for the remainder of the proof that every poorest row has \(m-1\) chips. Moreover, since \(\deg(D)\leq nm-2\), there must be at least two poorest rows.

By Lemma \ref{lemma:burning_complete}, any poorest row can be immediately burned unless it has all chips on one vertex, or all chips on different vertices; thus we may assume every poorest row has one of these forms.  Suppose for the moment that some poorest row \(r_i\) has all the chips on one vertex \(u\).  Running Dhar's algorithm from an unchipped vertex in that row burns every vertex in the row except for \(u\).  Let \(r_j\) be another poorest row.  Every vertex in \(r_j\) is adjacent to at least \(2\) vertices in \(r_i\), leading to at least two vertices in \(r_j\) burning, whence all of \(r_j\) burns except for at most one vertex (which could have \(m-1\) chips).  From here \(u\) will burn, since it has at least \(2\) neighbors in \(r_j\), one of which is burning.  Thus a poorest row burns.

We may now assume that every poorest row has a single chip on every vertex but one.  Now we turn to the unchipped adjacent vertices \(q\) and \(v\), which of necessity cannot both be in the same poorest row.  We deal with several cases; in each we run Dhar's burning algorithm from \(q\), meaning both \(q\) and \(v\) will burn.

\begin{itemize}
\item[(1)] The vertices \(q\) and \(v\) are in the same row $r_i$ (necessarily not a poorest one).  Assume for the moment $(n,m)\neq (4,4)$.
We first argue that the row $r_i$ has at most $2m-5$ chips.  Every row has at least \(m-1\) chips, meaning that the most chips a row could possibly have is \(nm-\alpha(TQ_{n,m})-(n-1)(m-1)=n+m-\alpha(TQ_{n,m})-1\).   If \(n\neq m\), this is equal to  \(n+m-\gcd(n,m)-1\).  If $\gcd(n,m)=2$, then we know $m\geq n+2$, and so \(n+m-\gcd(n,m)-1\leq 2m-5\); and if $\gcd(n,m)\geq 3$, we have  \(n+m-\gcd(n,m)-1\leq m+n-4\leq 2m-5\).  If $n=m$,  then \(m+n-\alpha(TQ_{n,m})-1\) is at most \(m+m-(m-2)-1=m+1\), which for $m\geq 6$ satisfies $m+1\leq 2m-5$.  The cases $n=m=2$ and $n=m=3$ can be ignored, as we have $\alpha(Q_{m,m})=1$ for $2\leq m\leq 3$; and for $n=m=5$ we have \(m+n-\alpha(TQ_{n,m})-1=5+5-5-1=4<5=2\cdot 5-5=2m-5\).

Now we argue, given that $r_i$ has at most $2m-5$ chips, that at least $m-1$ vertices in row $r_i$ burn.  Suppose not; then exactly $k$ burn for some $2\leq k\leq m-2$.  For the burning process to stabilize, each unburned vertex must have at least $k$ chips, so there are at least $k(m-k)$ chips on the unburned vertices.  Since the expression $k(m-k)$ is concave down in $k$, this expression is minimized at $2m-4$, when  $k\in\{2,m-2\}$.  This is more chips than are on the row, a contradiction.  Thus at least $m-1$ vertices in row $r_i$ burn.

With all of this row (sans perhaps one vertex) burning, most of a poorest row will burn (at least \(2\) burning edges coming from each vertex in \(q\)'s row), and then all of it will burn from its own internal fire.

Now we handle the case of $n=m=4$.  Since there are at least two poorest rows, each with $m-1=3$ chips, we know row $r_i$ must be adjacent to a poorest row, without loss of generality $r_{i+1}$ (as there are three other rows, and it is adjacent to two of them).  Regardless of the placement of $q$ and $v$ in $r_i$, they have two mutual neighbors in $r_{i+1}$, both of which will burn (each has at most one chip, and two incident burning edges).  The other two vertices in $r_{i+1}$ burn as well (again, each has at most one chip, now more than two incident burning edges).  Thus we have burned a poorest row.

\item[(2)] The vertices \(q\) and \(v\) are in the same column.  Running Dhar's will burn the poorest row's vertex in that column, since it has at most \(1\) chip.  At least one of \(q\) and \(v\) is not in that poorest row, and will have edges going to some of those vertices, meaning more vertices in the poorest row will burn, leading to the whole row burning.

\item[(3)] The vertices \(q\) and \(v\) are along the same diagonal.  The exact same argument from case (2) works here as well.
\end{itemize}
Having burned a poorest row in every case, we have reached our desired contradiction, completing the proof.

\end{proof}

We now prove our main result on toroidal queen's graphs.

\begin{proof}[Proof of Theorem \ref{theorem:toroidal}]Let \(D\) be a divisor on \(TQ_{n,m}\) with \(\deg(D)=\gon(Q_{n,m})\) and \(r(D)>0\).  By Proposition \ref{prop:independent_sets_toroidal}, \(D\) is equivalent to a divisor of the form \(\mathbb{1}_{S^C}\) for some maximal independent set \(S\), meaning that \(gon(TQ_{n,m})=\deg(D)=|S^C|=nm-\alpha(TQ_{n,m})\).

 For the second claim, given a maximum independent set \(S\) on \(TQ_{n,m}\), map \(S\) to the divisor class \([D]\) where \(D=\mathbb{1}_{S^C}\).  This is a rank \(1\) divisor of degree \(\gon(TQ_{n,m})\).  The argument that this gives a one-to-one and onto correspondence is identical to that from the proof of Theorem \ref{theorem:queens}.
 \end{proof}
 
\bibliographystyle{alpha}
\newcommand{\etalchar}[1]{$^{#1}$}

\end{document}